\documentclass[final]{l4dc2026}
\title[Robust Least-Squares Optimization for Data-Driven Predictive Control: A Geometric Approach]{Robust Least-Squares Optimization\\ for Data-Driven Predictive Control: A Geometric Approach}

\usepackage{times}
\usepackage{macros}

\author{%
 \Name{Shreyas Bharadwaj} \Email{shreyasnb@iitb.ac.in}\\
 \addr Indian Institute of Technology Bombay, India
 \AND
 \Name{Bamdev Mishra} \Email{bamdevm@microsoft.com}\\
 \addr Microsoft India%
 \AND
 \Name{Cyrus Mostajeran} \Email{cyrussam.mostajeran@ntu.edu.sg}\\
 \addr Nanyang Technological University, Singapore 
 \AND
 \Name{Alberto Padoan} \Email{apadoan@ece.ubc.ca}\\
 \addr University of British Columbia, Vancouver, Canada
\AND
\Name{Jeremy Coulson} \Email{jeremy.coulson@wisc.edu } \\
\addr University of Wisconsin-Madison, USA 
\AND
\Name{Ravi N Banavar} \Email{banavar@iitb.ac.in} \\
\addr Indian Institute of Technology Bombay, India
}

\DeclareMathOperator*{\argmax}{arg\,max}

\newcommand{\fin}[1]{{\color{black} #1 }}

\begin{document}
\maketitle

\begin{abstract}

The paper studies a geometrically robust least-squares problem that extends classical and norm-based robust formulations. Rather than minimizing residual error for fixed or perturbed data, we interpret least-squares as enforcing approximate subspace inclusion between measured and true data spaces. The uncertainty in this geometric relation is modeled as a metric ball on the Grassmannian manifold, leading to a min–max problem over Euclidean and manifold variables. The inner maximization admits a closed-form solution, enabling an efficient algorithm with a transparent geometric interpretation. Applied to robust finite-horizon linear–quadratic tracking in data-enabled predictive control, the method improves upon existing robust least-squares formulations, achieving stronger robustness and favorable scaling under small uncertainty. 
\end{abstract}


\section{Introduction}

Least-squares optimization is a ubiquitous problem in science and engineering. The problem is classically defined as 
\begin{equation*}
    \min_{x \in \R^k} \lVert Ax - b \rVert_2^2,
\end{equation*}
where $A \in \R^{n \times k}$ and $b \in \R^n$ are given data, and 
$x \in \R^k$ is the decision variable. 
In regression problems, the matrix $A$ represents a linear model and typically 
contains the regressors, while the vector $b$ collects the measured responses. 
The objective penalizes the residual $Ax - b$ in the Euclidean norm, yielding the 
coefficient vector $x$ that best fits the data in the least-squares sense.
The origins of least-squares trace back to celestial mechanics~\citep{gauss1809}, and the method has since become indispensable across scientific and engineering disciplines. In control theory, least-squares formulations are ubiquitous and underpin tasks such as identification, filtering, and optimal control. Despite its broad applicability, the least-squares solution is well known to be sensitive to perturbations in the data, especially when $A$ is ill-conditioned~\citep{golub2013matrix}. 
\newpage
\noindent The \textit{total least-squares} problem~\citep{golub1980analysis} generalizes classical least-squares by allowing perturbations in the individual matrix entries of both  $A$ and $b$. As in ordinary least-squares, the resulting solution may be sensitive to perturbations in the data~\citep{golub2013matrix}. Sensitivity to perturbations in the data $(A,b)$ can be mitigated in various ways.
One strategy is to consider the \textit{robust} least-squares problem
\begin{equation*} 
    \min_{x \in \R^{k}} \;
    \max_{A \in \mathbb{B}_{\rho}(\hat{A})}
    \|Ax - b\|_2^2 ,
\end{equation*}
where $\mathbb{B}_{\rho}(\hat{A})$ is the metric ball of radius $\rho$ centered at $\hat{A}$ defined by a given metric.
The specific choice of metric reflects the interpretation assigned to the data matrix, see, e.g., \citep{ghaoui97}, \citep{MARKOVSKY20072283}.



In this paper, we consider a robust least–squares formulation that incorporates the
\emph{geometric} nature of perturbations in linear models. In particular, we consider
\begin{equation} \label{eq:least_squares_robust_geometric}
\underset{x \in \R^{n}}{\textup{min}}\;
\underset{\mathcal{Y} \in \mathbb{B}^d_{\rho}(\hat{\mathcal{Y}})}{\textup{max}}\;
\|P_{\mathcal{Y}} x - b\|_2^2 ,
\end{equation}
where \(P_\mathcal{Y}\) denotes the orthogonal projector onto a \(k\)-dimensional subspace $\mathcal{Y}$ of $ \R^n$, and
\(\mathbb{B}^d_{\rho}(\hat{\mathcal{Y}})\) is a ball of radius \(\rho\) around \(\hat{\mathcal{Y}}\) in the Grassmannian
\(\Gr(k,n)\), i.e., the manifold of \(k\)-dimensional subspaces of \(\R^n\). The problem builds on the formulation introduced in~\citep{coulson2025}, and is motivated by a wide range of applications, including robust subspace tracking in signal processing~\citep{delmas}, computer vision~\citep{He2012}, and system identification~\citep{sasfi2024}. 

Our motivating application arises in data–driven predictive control~\citep{jeremy2019}, where uncertainty fundamentally enters through the finite–horizon behavior of a linear time-invariant (LTI) system. In the behavioral approach~\citep{willems2005}, the finite–horizon trajectories of an LTI system form a finite–dimensional subspace and can be characterized as the image of a Hankel matrix constructed from data, by the \emph{fundamental lemma}~\citep{willems2005}. Hence, a range of data–driven predictive control methods — in both deterministic~\citep{jeremy2019,de2019formulas,berberich2022linear} and stochastic settings~\citep{breschi2023datadriven} — rely on dynamic constraints expressed as subspaces, which are naturally modeled as points on the \emph{Grassmannian}. As a consequence, uncertainty in the data translates directly into subspace uncertainty. This perspective has proved effective in mode recognition, fault detection, and system identification~\citep{padoan2022,padoan2025distances}. In line with these findings, we show in Section \ref{sec:solution} that \eqref{eq:least_squares_robust_geometric} captures this geometric uncertainty while preserving favorable structure and providing robustness guarantees for data–driven predictive control.

\paragraph{Related work.}  Robust extensions of least-squares problem have been explored under various perturbation models, including norm-bounded uncertainty~\citep{ghaoui97}, total least--squares~\citep{golub1980analysis},  and geometric formulations on matrix manifolds~\citep{absil}, the latter being closest to the setting considered here. In this work, we build on~\citep{coulson2025} and study a constrained variant of the geometric robust least--squares problem, which requires dedicated analysis due to the interaction between the constraints and the subspace uncertainty set. Robust formulations for model predictive control (MPC) trace back to the seminal work of~\citep{mayne2005robust} and have been extensively studied for linear systems, including in data--driven settings~\citep{scokaert1998,xie2024,xie2025,xie2026}.  Data–driven predictive control methods often rely on behavioral formulations in which feasible trajectories lie in a data–determined subspace, both in deterministic~\citep{jeremy2019,de2019formulas,berberich2022linear} and in stochastic~\citep{breschi2023datadriven} settings (see also~\cite{markovsky2021} for an overview). Related developments include robust extensions of direct data--driven MPC~\citep{huang2021}, and unifying perspectives that connect direct (behavioral) and indirect (model--based) schemes via regularization and relaxation~\citep{dorfler2023}.  Closest to our setting are~\citep{xie2024,xie2025,xie2026}, which also address min--max data--driven MPC. Their formulations rely on linear matrix inequalities (LMIs) and require knowledge of the system order and state dimensions. In contrast, our approach remains order--free, avoids explicit state parameterization, and yields a scalable alternative to LMI--based designs.

\paragraph{Contributions.} 
The article presents a tractable and geometrically grounded formulation of the robust least-squares problem arising in data-driven predictive control. Building on the extended abstract~\citep{coulson2025}, uncertainty is modeled as a metric ball on the Grassmannian manifold, leading to a min–max optimization problem that captures geometric robustness in a principled way. Our main contribution, in contrast to \citep{coulson2025}, is an analytic, closed-form solution to the inner maximization, which transforms the original nonconvex problem into a convex and computationally efficient formulation. The resulting algorithm constitutes a marked improvement over the algorithm introduced in~\citep{coulson2025}, offering significantly faster convergence and lower computational cost. This novel tractability allows us to apply the methodology to data-driven predictive control, where algorithm scalability is vital. Compared with data-driven MPC~\citep{berberich2021}, our method achieves similar tracking performance while providing a clear system-theoretic interpretation of robustness in terms of finite-horizon trajectories.

\section{Preliminaries}\label{sec:prelims}

     The \textit{Grassmannian} $\Gr(k,n)$ is the Riemannian manifold of all $k$-dimensional linear subspaces of $\R^n$~\citep{boumal2023}. A standard way to describe it is through the \textit{Stiefel manifold}
$
    \St(k,n)
    := \bigl\{\, Y \in \R^{n \times k} : Y^\top Y = I_k \,\bigr\},
$
whose elements are matrices with orthonormal columns. 
An alternative convenient representation uses orthogonal projection matrices. Every $k$-dimensional subspace $\mathcal{Y}$ of $\R^n$ admits a unique orthogonal projector $P_{\mathcal{Y}}$, which makes this identification natural. In this representation, the Grassmannian can be written as
$
    \Gr(k,n)
    =
    \bigl\{
        P \in \R^{n \times n}
        :\;
        P^\top = P,\;
        P^2 = P,\;
        \mathrm{rank}(P) = k
    \bigr\},
$
where each point corresponds to a unique symmetric, idempotent matrix of rank~$k$, and  the associated $k$-dimensional subspace is given by the image of~$P$. The link between these two representations is given by the canonical projection $\pi:\St(k,n)\to\Gr(k,n)$ such that $\pi(Y)=YY^\top$, which maps an orthonormal basis $Y$ to the orthogonal projector onto its span. Given a smooth function $f : \Gr(k,n) \to \R$, its Riemannian gradient at $\Sy$ is the orthogonal projector of the Euclidean gradient of any smooth extension $\bar f : \R^{n\times k} \to \R$ onto the tangent space:
\[
\grad \ f \, (\mathcal{Y})
= P_{\mathcal{Y}}^\perp \,\nabla \bar f(Y),
\qquad
P_{\mathcal{Y}}^\perp := I_n - YY^\top,
\]
where $\bar f$ satisfies $\bar f(Y) = f(\mathcal{Y})$ for any orthonormal basis $Y$ of $\mathcal{Y}$, and $\nabla \bar f(Y)$ denotes the standard Euclidean gradient of $\bar f$.

Several distances can be used to compare subspaces~\citep{boumal2023}, such as geodesic, projection, and Procrustes distances. The \emph{chordal distance} is particularly convenient because it admits a simple expression in terms of projection matrices and is computationally efficient. It is defined as 
\begin{equation} \label{eq:chordal}
    d(\mathcal{Y}_1,\mathcal{Y}_2) := \sqrt{\mathrm{Tr}\!\left(P_{\mathcal{Y}_1}^\perp P_{\mathcal{Y}_2}\right)}.
\end{equation} 

\subsection{Behaviors and data-driven representations of dynamical systems}

Consider the discrete-time LTI state-space system
\begin{equation}\label{eq:LTI}
x(t+1)=Ax(t)+Bu(t),\qquad y(t)=Cx(t)+Du(t),
\end{equation}
where $u(t)\in\R^m$ is the input, $y(t)\in\R^p$ is the output, and $x(t)\in\R^{n_s}$ is the state for $t\in\N$. The input-output behavior~\citep{willems1986} of the state-space system~\eqref{eq:LTI} is defined as
\begin{equation}\label{eq:behavior}
\Be:=\{w=(u,y):\N\to\R^{m+p}\mid \exists \, x:\N\to\mathbb{R}^{n_s} \text{ such that } \eqref{eq:LTI} \text{ holds for all } t \in \N\}.
\end{equation}
Given $\Be$ as above, we call~\eqref{eq:LTI} a state-space representation of $\Be$ and elements of $\Be$ \emph{trajectories}. We assume without loss of generality that the state-space system~\eqref{eq:LTI} is \emph{minimal}, in the sense that $n_s$ is as small as possible across all state-space representations of $\Be$. The \emph{lag} of $\Be$ is defined as the smallest integer $\ell$ such that $\rank [C^\top \; A^\top C^\top \; \cdots \; (A^{\ell-1})^{\top}C^\top]^\top=\rank [C^\top \; A^\top C^\top \; \cdots \; (A^{\ell})^{\top}C^\top]^\top$.

The restriction of a trajectory \( w : \N \to \R^{m+p} \) to a nonempty interval \( [i,j] \) is defined as $w|_{[i,j]} = \col (w(i),  w(i+1),  \cdots , w(j)),$ where \(\operatorname{col}(w_1,\ldots,w_k) := [\, w_1^\top \; \cdots \; w_k^\top \,]^\top\). The \textit{restricted} behavior is defined as
$
\Be|_{[1,L]}:= \{w|_{[1,L]} \mid w \in \Be \}.
$

We define the \textit{Hankel matrix} of depth $L \in \N$, associated with a vector $w|_{[1,T]} \in \R^{qT}$ as
\[
    \Ha_L(w|_{[1,T]}) = \begin{bmatrix}
        w(1) & w(2)  & \cdots &  w(T-L+1)   \\
 w(2) & w(3)  & \cdots &   w(T-L+2)   \\
 \vdots  & \vdots  & \ddots & \vdots  \\
 w(L) & w(L+1)  & \cdots  & w(T)
    \end{bmatrix}.
\]
The restricted behavior of LTI system~\eqref{eq:LTI} is a subspace, whose dimension is determined only by its state-space dimension $n_s$, lag $\ell$, and time horizon $[1,L]$.

\begin{lemma}\cite[Lemma 2.1]{dorfler2023}
\label{lemma-rank}
    Consider the LTI system~\eqref{eq:LTI} with associated behavior $\Be$. The restricted behavior $\Be |_{[1,L]}$ is a linear subspace of $\R^{qL}$. Moreover, for $L \geq \ell$, $\dim \Be|_{[1,L]} = mL + n_s$.
\end{lemma}
As a direct consequence of Lemma \ref{lemma-rank}, the restricted behavior can be represented as the image of a raw data matrix. A version of this result known as the \textit{fundamental lemma} \citep{willems2005} is presented below.

\begin{lemma}~\cite[Corollary 19]{markovsky2023id}
\label{lemma-fund}
    Consider the LTI system~\eqref{eq:LTI} with associated behavior $\Be$. Let $L > \ell$, $T\geq L$, and $w|_{[1,T]} \in \Be |_{[1,T]}$. Then, 
    $\Be |_{[1,L]} = \Image \Ha_L(w|_{[1,T]}), 
    $ if and only if
    $\rank \Ha_L(w|_{[1,T]}) = mL + n_s$.
\end{lemma}
This condition is referred to as the \textit{generalized persistency of excitation} (GPE) condition.

\subsection{Finite-horizon linear quadratic tracking}
\label{sec:prediction}
Given an LTI system, an initial trajectory compatible with its behavior, and a quadratic cost, the linear quadratic tracking problem is to find a trajectory that is as close as possible to a given reference. Over a finite-horizon, the problem can be formulated as follows~\citep{markovsky2008}: Consider the system~\eqref{eq:LTI} with associated behavior $\Be$, an initial trajectory ${w_{\textrm{ini}} \in \Be|_{[1,T_{\textrm{ini}}]}}$, with $T_{\textrm{ini}} \ge 1$, a reference ${w_{\textrm{ref}} \in \R^{qT_{\textrm{f}}}}$, with $T_{\textrm{f}}\ge 1$, and a symmetric positive definite matrix ${\Phi\in\R^{q\times q}}$. The \textit{linear quadratic tracking (LQT) problem} is to find $w_{\textrm{f}}$ that minimizes the cost
 \[ \label{eq:cost_quadratic}
 \norm{w_{\textrm{f}}-w_{\textrm{ref}}}_{I \otimes \Phi}^2 = \sum_{t=1}^{T_{\textrm{f}}} (w_{\textrm{f}}(t)-w_{\textrm{ref}}(t))^{\top} \Phi (w_{\textrm{f}}(t)-w_{\textrm{ref}}(t)) ,
 \]
and has a \textit{past constraint} on the initial trajectory $w_{\textrm{ini}}$, i.e.,
\begin{equation}\label{eq:linear-quad-track}
    \begin{aligned}
        \underset{w_{\mathrm{f}} \in \R^{qT_{\mathrm{f}}}}{\min} 
        & \quad \bigl\| w_{\mathrm{f}} - w_{\mathrm{ref}} \bigr\|_{I \otimes \Phi}^2 , \quad 
        \textup{subject to} 
        & \quad \operatorname{col}(w_{\mathrm{ini}}, w_{\mathrm{f}}) 
        \in \Be\big|_{[1,\,T_{\mathrm{ini}} + T_{\mathrm{f}}]}.
    \end{aligned}
\end{equation}
  If a parametric model of the system behavior is available -- for example, a state-space representation -- ~\eqref{eq:linear-quad-track} gives rise to a classical finite-horizon optimal control problem \citep{anderson2007}.

\section{Proposed Geometric Modeling}\label{sec:solution}

The finite–horizon LQT problem~\eqref{eq:linear-quad-track} admits a natural least–squares reformulation: 
the cost in~\eqref{eq:linear-quad-track} is quadratic, and the requirement that 
$\col(w_{\mathrm{ini}}, w_{\mathrm{f}}) \in\Be|_{[1,T_{\textrm{ini}}+T_{\textrm{f}}]}$ is a subspace constraint. Consequently, identifying the restricted behavior $\Be|_{[1,T_{\textrm{ini}}+T_{\textrm{f}}]}$ with a subspace $\hat{\Sy} \in \Gr(mL+n_s, qL)$, any feasible trajectory $w\in\Be|_{[1,T_{\textrm{ini}}+T_{\textrm{f}}]}$ of problem~\eqref{eq:linear-quad-track} satisfies
\[
w = P_{\hat{\Sy}} x,\text{ for some } x \in \R^{qL},
\]
where $L=T_{\textrm{ini}}+T_{\textrm{f}}$, along with the affine constraint enforcing consistency with the observed past,
\[
Mw = w_{\mathrm{ini}}, \quad \text{where} \ M := 
\begin{bmatrix}
I_{qT_{\mathrm{ini}}} & 0
\end{bmatrix}
\in \R^{qT_{\mathrm{ini}} \times qL}.
\]
Substituting these expressions in~\eqref{eq:linear-quad-track}
shows that finite--horizon LQT reduces to a constrained least--squares problem 
\begin{equation}\label{eq:ls-behavioral}
\min_{x \in \R^{qL}} \;\|P_{\hat{\Sy}}x - b\|_2^2
\quad \text{s.t.} \quad MP_{\hat{\Sy}}x = Mb,
\end{equation}
where $b := \col(w_{\mathrm{ini}}, w_{\mathrm{ref}})$. In practice, the subspace $\Syhat$ is estimated from finite, noisy data, motivating a
robust formulation. Motivated and inspired by~\citep{coulson2025}, one may allow the behavior subspace to vary within a ball
$\B_{\rho}^{d}(\hat{\Sy})$, yielding
\begin{equation}\label{eq:constraint}
\begin{aligned}
    \min_{x \in \R^{qL}} \;& \max_{\Sy \in \B_{\rho}^{d}(\hat{\Sy})}
    \|P_{\Sy} x - b\|_2^2, \quad
    \text{s.t.}\;& MP_{\Sy}x = Mb,
\end{aligned}
\end{equation}
where the inner maximization accounts for uncertainty in the system’s finite–horizon behavior. The radius $\rho$ quantifies uncertainty in the estimated subspace $\Syhat$ due to finite, noisy data. Concretely, if trajectories are corrupted by noise with noise-to-signal ratio $\sigma$, subspace tracking algorithms \citep{sasfi2024} yield finite-sample bounds on $d(\Sy, \Syhat)$, providing a principled choice of $\rho \approx \sin(\theta)$ where $\theta$ is the angular deviation between the true and estimated subspace. This makes $\rho$ directly interpretable as a geometric perturbation, in contrast to the slack variable in \citep{berberich2021} which has no such system-theoretic meaning.

Directly enforcing the constraint in~\eqref{eq:constraint} leads to a projected min–max problem
on the Grassmannian, which is prohibitively expensive for online control. To obtain a tractable surrogate, we relax the constraint via Tikhonov regularization \citep{golub1999}, leading to the optimization problem
\begin{equation}\label{eq:reg-robust-lsq}
\min_{x \in \R^{qL}} \;\max_{\Sy \in \B_{\rho}^{d}(\hat{\Sy})}
\|P_{\Sy}x - b\|_2^2
+ \gamma \|MP_{\Sy}x - Mb\|_2^2,
\end{equation}
where $\gamma > 0$ penalizes inconsistency with past data. As $\gamma \to \infty$, solutions of \eqref{eq:reg-robust-lsq} converge to those of \eqref{eq:constraint} in the sense of standard quadratic penalty methods \citep[Ch. 4]{bertsekas_nlp}: the penalty term forces $MP_{\Sy} x \to Mb$ asymptotically. Finite $\gamma$ introduces a bias, and empirical choices of $\gamma$ (as in Section \ref{sec:results}) yield accurate constraint satisfaction for the systems considered. In practice, we initialize $\gamma$ at a moderate value and increase until constraint violation falls below a threshold.


\subsection{Main results}

We now present the solution methodology to the problem \eqref{eq:reg-robust-lsq}. The following results show that the inner problem has an explicit solution to $\Sy$ due to our choice of the chordal distance for the ball constraint.

\begin{lemma}\label{lemma:1}
    Consider the robust least-squares problem~\eqref{eq:reg-robust-lsq}. Let $\hat{Y} \in \R^{n \times k}$ and  ${\hat{\Sy} \in \Gr(k,n)}$  be given such that $P_{\hat{\Sy}} = \pi(\hat{Y})$. 
    Define $f: \R^n \times \Gr(k,n) \to \R_{\geq 0}$ as
    \begin{equation}\label{eq:reg-cost}
        f(x,\Sy) := \lVert P_{\Sy} x - b \rVert_2^2 + \gamma \lVert MP_{\Sy}x - Mb \rVert_2^2.
    \end{equation}
    Then the following hold.
    \begin{enumerate}
        \item $f(x,\Sy^*(x))$ is convex in $x$, where $\Sy^*(x) = \argmax_{\Sy \in \B_{\rho}^d(\hat{\Sy})} f(x,\Sy)$.
        \item The gradient of $f$ with respect to $x$ is given by
        \begin{equation}
        \label{eq:gradfx}
            \nabla_x f(x,\Sy^*(x)) = 2P_{\Sy^*(x)}(x-b) + 2\gamma P_{\Sy^*(x)}M^\top M (P_{\Sy^*(x)}x - b ).
        \end{equation}
    \end{enumerate}
\end{lemma}

\begin{proof}
 Refer to Supplementary for a detailed proof.
\end{proof}

Among the available Grassmannian distance metrics--chordal, gap, and related variants--the key consideration is tractability, i.e., whether the inner problem in $\Sy$ admits a closed-form solution, as established in the following result. Anticipating the structure of the cost function in~\eqref{eq:equiv-reg-rob}, we adopt the chordal metric $d$  defined in~\eqref{eq:chordal} whose trace representation leads to an explicit and efficient formulation. However, other metrics may also yield tractable variants of the problem and represent a promising direction for future investigation.

\begin{theorem}\label{theorem:1}
Consider the robust least-squares problem~\eqref{eq:reg-robust-lsq}. Let $\Sy^*(x) = \argmax_{\Y \in \B_{\rho}^{d}(\hat{\Sy})} f(x,\Sy),$
where $f$ is defined in \eqref{eq:reg-cost} and $Y^* \in \St(k,n)$ such that $\pi(Y^*(x)) = P_{\Sy^*(x)}$. Then, there exists $\lambda^* \geq 0$ such that the maximizer $\Sy^*(x)$ of the inner problem satisfies first-order optimality conditions, and 
\[
P_{\Sy^*(x)} = \pi(Y^*(x)),
\qquad  Y^*(x) = \{ \textrm{top--k eigenvectors of } B(x, \lambda^*; \gamma) \},
\]
where 
$
B(x, \lambda^*; \gamma) = xx^\top - xb^\top - bx^\top + \lambda^* \hat{Y}\hat{Y}^\top + \gamma(xx^\top - M^\top Mbx^\top - x b^\top M^\top M) .
$ 
\end{theorem}

\begin{proof}
    The proof proceeds by reformulating the inner maximization into a sequence of equivalent problems, the last of which admits an explicit analytic solution. It can be shown that the inner problem in~\eqref{eq:reg-robust-lsq} can be equivalently expressed as 
    \begin{equation}\label{eq:equiv-reg-rob}
        \max_{Y \in \St(k,n)} \Tr\left(Y^\top B(x,\lambda;\gamma)Y \right),
    \end{equation}
    where $B(x,\lambda; \gamma) := A(x;\gamma) + \lambda \hat{Y} \hat{Y}^\top$ for $\lambda \geq 0$ such that $A(x;\gamma) := xx^\top - xb^\top - bx^\top + \gamma(xx^\top - M^\top M bx^\top - xb^\top M^\top M)$. \eqref{eq:equiv-reg-rob} has a well-known closed-form solution \citep[Section 2.1]{absil} given by $Y^*(x) = \{\textrm{top--$k$ eigenvectors of }B(x, \lambda^*; \gamma) \}$.
The corresponding optimal subspace $\Sy^*(x)$ is such that $P_{\Sy^*(x)} = \pi(Y^*(x))$. The detailed proof is in the Supplementary.
\end{proof}

\begin{remark}[Characterization of $\lambda^*$]
The scalar $\lambda^* \geq 0$ is the optimal dual variable (Lagrange multiplier) for the ball constraint $g(\Sy)=d^2(\Sy,\hat{\Sy})-\rho^2 \leq 0$. KKT complementary slackness \cite[Def.~2.3]{liu2019} gives $\lambda^* g(\Sy^*) = 0$.
    Thus, if $\lambda^* = 0$, the constraint $g(\Sy^*) < 0$ must be satisfied, but if $\lambda^* > 0$, then $g(\Sy^*) = 0$, i.e., the optimal subspace $\Sy^*$ must lie on the boundary of the ball $\B^{d}_{\rho}({\hat{\Sy}})$. It can also be shown that there exists a threshold $\lambda^{**}$ such that for $\lambda^* \geq \lambda^{**}$, the top-$k$ eigenspace of $B(x, \lambda^*; \gamma)$ is unique; see Supplementary.
    \end{remark}
The above remark implies that if the optimal subspace $\Sy^*$ is such that $g(\Sy^*)>0$, then $\lambda^*$ must be larger so that $\Sy^*$ lies on the boundary of the uncertainty ball (worst-case scenario). Analyzing the matrix $B(x,\lambda; \gamma) = A(x;\gamma) + \lambda \hat{Y}\hat{Y}^\top$, we observe that since $Y^*(x)$ is the basis for the top-$k$ eigenspace of $B(x,\lambda^*; \gamma)$, increasing $\lambda$ implies that $\Sy^*(x)$ gets closer to $\hat{\Sy}$, in the sense that $\Sy^*(x)$  depends implicitly on $\lambda$, and the optimal subspace is such that $d(\Sy^*(x,\lambda^*),\hat{\Sy}) = \rho$. Geometrically, lower data uncertainty corresponds to a smaller radius~$\rho$ and, consequently, a larger multiplier~$\lambda^*$, which constrains $\Sy^*(x,\lambda^*)$ to remain close to~$\hat{\Sy}$.

Exploiting the explicit solution proposed to the inner problem in the above results, the min-max optimization simplifies to a convex minimization framework, which can be easily solved with standard methods, such as gradient-based schemes. For example, at each step $i$, one may compute the gradient $\nabla_{x} f(x_i, \Sy^*(x_i))$ using \eqref{eq:gradfx}, where $\Sy^*(x_i)$ has an explicit solution given in Theorem~\ref{theorem:1}. Thus, we propose the optimization scheme in Algorithm \ref{alg:cons-robust-ls} for solving the robust least-squares problem, implemented via standard gradient descent scheme.
For a fixed subspace $\Sy^*(x)$ computed at each iterate $x_i$ (i.e., treating $\Sy^*(x_i)$ as constant at step $i$), the gradient $\nabla_x f(x_i, \Sy^*(x_i))$ is Lipschitz continuous in $x$ with constant $L_f = 2 + 2\gamma$. This justifies convergence of Algorithm \ref{alg:cons-robust-ls} in the sense that the gradient norm $\norm{\nabla_x f(x_i, \Sy^*(x_i))} \to 0$, as each outer step is a descent step on the instantaneous cost, and standard results guarantee $\|\nabla_x f\| \to 0$ along the outer iterates under mild assumptions.
Standard optimization theory \citep[Chapter 3]{nocedal2006} states that Algorithm \ref{alg:cons-robust-ls} converges to a minimizer $x^*$ for any fixed step-size $\alpha$, satisfying $0 < \alpha < 2/{L_f}$, which simplifies to $0<\alpha<\frac{1}{1+\gamma}$ in our case. 
 
 The subspace optimization problem presented in this section directly connects to the linear quadratic tracking objective described in Section~\ref{sec:prediction}. In line with the model predictive control (MPC) paradigm, Algorithm~\ref{alg:cons-robust-ls} is solved at each time step~$k$ in a receding-horizon fashion. From the resulting optimal trajectory $w^* = P_{\Sy^*}x^* = \col(w^*_{\textrm{ini}}, w^*_{\textrm{f}})$, the first control input of $w^*_{\textrm{f}}$ is extracted and applied to the system, and this is repeated for all $t \in [1,T]$.
\begin{algorithm2e}[t]
        \Data{Subspace Estimate $\hat{Y} \in \R^{n \times k}$, Reference $b \in \R^n$, Cost $f$.}\\
        \Input{Initial guess $x_0 \in \R^n$, step-size $\alpha > 0$, $\rho, \gamma > 0$, \texttt{tolx} $> 0$, $M \in \R^{l \times n}$.}\\
        \Output{Optimal subspace $Y^*$, Optimal point $x^*$.}
    \BlankLine\For{$i=0,1,2,\cdots,$}
    {
            $A_i = x_i x_i^\top - x_i b^\top - bx_i^\top+\gamma (x_i x_i^\top - M^\top Mb x_i^\top - x_i b^\top M^\top M)$ \;
            $Y^* =\texttt{eigs}(A_i,k)$\;
            \eIf{$d(\pi(Y^*),\pi(\hat{Y}))\leq \rho$}{ 
                $\lambda_i^* = 0$\;
            }{
            Find $\lambda_i^* > 0$ such that $ | d (\pi(\texttt{eigs} (A_i + \lambda_i^* \hat{Y}\hat{Y}^\top, k)), \pi(\hat{Y})) - \rho |=0$\;
            $B_i = A_i + \lambda_i^* \hat{Y}\hat{Y}^\top$\;
            $Y^* =\texttt{eigs} (B_i,k)$\;
            }   
        \eIf{$\lVert \nabla_{x}f(x_i, \pi(Y^*(x_i))) \rVert \leq$ \texttt{tolx}}{
        $x^* = x_i$\;
        break\;
        }{ 
        $x_{i+1} = x_i - \alpha \nabla_{x}f(x_i, \pi(Y^*(x_i)))$; 
        }
    }
\caption{Proposed Constrained Robust Least-Squares Optimization Algorithm}
\label{alg:cons-robust-ls}
\end{algorithm2e}

\section{Numerical Simulations}
\label{sec:results}

We present numerical results for the proposed algorithm applied to the finite-horizon linear quadratic tracking problem introduced in Section~\ref{sec:prediction}.  We consider both tracking and regulation tasks in two case studies. In each case, the system dynamics are described by an LTI state-space model
\begin{equation} \label{eq:lti}
x(t+1) = Ax(t) + Bu(t), \qquad 
y(t) =Cx(t) + Du(t) + e(t),
\end{equation}
where $u(t)\in\R^m$ is the input, $y(t)\in\R^p$ is the output, and $x(t)\in\R^n$ 
is the state, and $e(t)\in\R^p$ is the measurement noise, at time $t\in\N$. Output measurement noise is chosen as it directly corrupts the observed trajectories used for subspace identification, making its effect on $\rho$ transparent and easy to analyze. Process noise and innovations representations are equivalent under suitable transformations and are left as extensions. We compare the performance of the proposed algorithm with (a) a nominal (non-robust $(\rho=0)$ and noiseless) controller \eqref{eq:ls-behavioral} assuming a nominal $\hat{\Sy}$, and then with (b) a robust data-driven MPC (DD--MPC) from \citep{berberich2021}. 
A key distinction is that $\rho$ carries a direct geometric interpretation (Section \ref{sec:solution}), in contrast to the slack variable $\lambda_{\alpha} \bar{\epsilon}$ in \cite{berberich2021}, which is tuned heuristically.
The codes and utilities to reproduce the experiments are available at \url{https://github.com/shreyasnb/Robust-least-squares-for-data-driven-control}.

\subsection{Double integrator}
Consider the system \eqref{eq:lti} with $A = \begin{bmatrix}
        1 & 0.5 \\ 0 & 1
    \end{bmatrix}, B = \begin{bmatrix}
        0.125 \\ 0.5
    \end{bmatrix}, C = \begin{bmatrix}
        1 & 0
    \end{bmatrix}, D = 0$. The reference to track is $w_{\textrm{ref}}(t)=\col{(0,1)}$. Figure \ref{fig:dint1} shows the tracking in the presence of measurement noise with noise-to-signal ratio $\sigma = 0.1, 0.2$ in offline subspace identification of $\hat{\Sy}$ (using Lemmas~\ref{lemma-rank},\ref{lemma-fund}), as well as in online control. The tracking performance of the proposed method is consistent in both cases for $\gamma \geq 4$, although very large values are undesirable since they have a negative impact on the optimization performance. It is also crucial to note that similar tracking performance is observed in $\sigma=0.1$ case, for ball radius $\rho \in (0, \sin(2.5^\circ)]$, beyond which we observe significant (bounded) oscillations around the reference $r=1$, instead of asymptotic convergence as in Figure~\ref{fig:dint1}. Similarly, for $\sigma=0.2$, we have $\rho \in (0,\sin(4^\circ)]$. The radius $\rho$ is selected by measuring $d(\hat{\Sy}, \Sy_{\textrm{true}})$ on offline data for a given $\sigma$.

We observe similar tracking performance between the proposed algorithm and DD--MPC for $L = 35$, subspaces in $\Gr(37,70)$, and other design parameters (see Supplementary). Upon increasing the noise level from $\sigma=0.1$ to $0.2$, we relax our uncertainty radius from $\rho=\sin(1^\circ)$ to $\sin(2^\circ)$, while the slack variable $\lambda_{\alpha}\bar{\epsilon}$ from \citep{berberich2021} would need to be increased from $0.2$ to $0.5$ to observe similar controller performance (keeping all other parameters constant).
\begin{figure}[h]
  \begin{center}
    \begin{tikzpicture}
      \begin{axis}[
          width=0.95\linewidth,
          height=0.3\linewidth,
          xlabel={},
          xticklabels={},
          ylabel=$y(t)$ ,
          xmin = 0, xmax = 30,
          axis x line* = top,
          axis y line* = left,
          legend style={at={(0.23,1.05)}, anchor=south, legend columns=-1}
        ]
        \addplot[mark=square, mark size = 1, red, thick] table[x=t,y=ycl, col sep=comma]{data/case1_gamma4_rho1.csv};
        \addlegendentry{$y^{\textrm{PROP}}(t)$}
        \addplot[mark=square, mark size = 1, orange!90, thick] table[x=t,y=ycl, col sep=comma]{data/case1_mpc_sigma0p1.csv};
\addlegendentry{$y^{\textrm{DD--MPC}}(t)$}
        \addplot[mark=none, mark size=1, black, ultra thick] table[x=t,y=ycl_nl, col sep=comma]{data/case1_gamma4_noiseless.csv};
        \addlegendentry{$y^{\textrm{NOM}}(t)$}
        \addplot [mark=none, red, dashed, thick]
        table[x=t,y=r,col sep=comma]{data/case1_gamma4_rho1.csv};
        \draw[step=0.5cm,gray!30,thin] (0,-3) grid (30,2);
        \addplot [name path=upper,draw=none] table[x=t,y expr=\thisrow{ycl}+0.1,col sep=comma]{data/case1_gamma4_rho1.csv};
        \addplot [name path=lower,draw=none] table[x=t,y expr=\thisrow{ycl}-0.1,col sep=comma]{data/case1_gamma4_rho1.csv};
        \addplot [fill=red!20] fill between[of=upper and lower];
         \addplot [name path=upper,draw=none] table[x=t,y expr=\thisrow{ycl}+0.1,col sep=comma]{data/case1_mpc_sigma0p1.csv};
        \addplot [name path=lower,draw=none] table[x=t,y expr=\thisrow{ycl}-0.1,col sep=comma]{data/case1_mpc_sigma0p1.csv};
        \addplot [fill=orange!30] fill between[of=upper and lower];
      \end{axis}
        \begin{axis}[
          width=0.95\linewidth,
          height=0.3\linewidth,
          xlabel={},
          xticklabels={},
          ylabel=$u(t)$ ,
          xmin = 0, xmax = 30,
          axis x line* = bottom,
          axis y line* = right,
          legend style={at={(0.78,1.05)}, anchor=south, legend columns=-1}
        ]
        \addplot[mark=*, mark size = 0.5, blue, thick] table[x=t,y=ucl, col sep=comma]{data/case1_gamma4_rho1.csv};
        \addlegendentry{$u^{\textrm{PROP}}(t)$}
        \addplot[mark=*, mark size=0.5, cyan!90, thick] table[x=t,y=ucl, col sep=comma]{data/case1_mpc_sigma0p1.csv};
        \addlegendentry{$u^{\textrm{DD--MPC}}(t)$}
        \addplot[mark=none, blue!90, ultra thick] table[x=t,y=ucl_nl, col sep=comma]{data/case1_gamma4_noiseless.csv};
        \addlegendentry{$u^{\textrm{NOM}}(t)$}
      \end{axis}  
    \end{tikzpicture}
    \bigskip
    \vspace{-1.1cm}
    \begin{tikzpicture}
      \begin{axis}[
          width=0.95\linewidth,
          height=0.3\linewidth,
          xlabel={},
          ylabel=$y(t)$ ,
          xticklabels={},
          xmin = 0, xmax = 30,
          axis x line* = top,
          axis y line* = left,
          legend style={at={(0.8,0.75)}, fill=white, fill opacity=0.6, draw opacity=1,text opacity=1}
        ]
        \draw[step=0.5cm,gray!30,thin] (0,-3) grid (30,2);
        \addplot[mark=square, mark size = 1, red, thick] table[x=t,y=ycl, col sep=comma]{data/case1_gamma4_rho2.csv};
        \addplot[mark=square, mark size=1, orange!90, thick] table[x=t,y=ycl, col sep=comma]{data/case1_mpc_sigma0p2.csv};
        \addplot [mark=none, red, dashed, thick]
        table[x=t,y=r,col sep=comma]{data/case1_gamma4_rho2.csv};
        \addplot[mark=none, black, ultra thick] table[x=t,y=ycl_nl, col sep=comma]{data/case1_gamma4_noiseless.csv};
        \addplot [name path=upper,draw=none] table[x=t,y expr=\thisrow{ycl}+0.2,col sep=comma]{data/case1_gamma4_rho2.csv};
        \addplot [name path=lower,draw=none] table[x=t,y expr=\thisrow{ycl}-0.2,col sep=comma]{data/case1_gamma4_rho2.csv};
        \addplot [fill=red!20] fill between[of=upper and lower];
        \addplot [name path=upper,draw=none] table[x=t,y expr=\thisrow{ycl}+0.2,col sep=comma]{data/case1_mpc_sigma0p2.csv};
        \addplot [name path=lower,draw=none] table[x=t,y expr=\thisrow{ycl}-0.2,col sep=comma]{data/case1_mpc_sigma0p2.csv};
        \addplot [fill=orange!30] fill between[of=upper and lower];
      \end{axis}
     \begin{axis}[
          width=0.95\linewidth,
          height=0.3\linewidth,
          xlabel=$t$,
          ylabel=$u(t)$ ,
          xmin = 0, xmax = 30,
          axis x line* = bottom,
          axis y line* = right,
          legend style={at={(0.98,0.75)}, fill=white, fill opacity=0.6, draw opacity=1,text opacity=1}
        ]
        \addplot[mark=*, mark size = 0.5, blue, thick] table[x=t,y=ucl, col sep=comma]{data/case1_gamma4_rho2.csv};
        \addplot[mark=*, mark size=0.5, cyan!90, thick] table[x=t,y=ucl, col sep=comma]{data/case1_mpc_sigma0p2.csv};
        \addplot[mark=none, blue!90, ultra thick] table[x=t,y=ucl_nl, col sep=comma]{data/case1_gamma4_noiseless.csv};
      \end{axis}  
    \end{tikzpicture}
  \end{center}
  \caption{Comparison of proposed method (PROP) with DD--MPC and a nominal controller (NOM), of tracking for double-integrator system with reference $r=1$ (red dashed horizontal line). Top: trajectories for noise level $\sigma = 0.1$. Bottom: trajectories for noise level $\sigma = 0.2$.}
  \label{fig:dint1}
\end{figure}
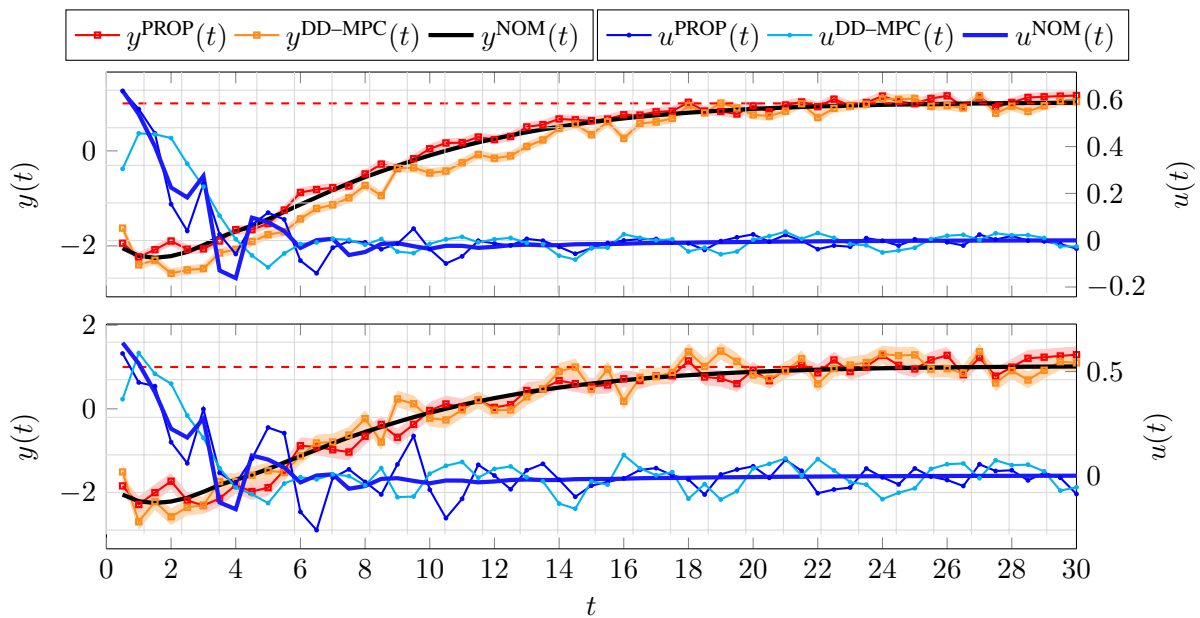
\vspace{-5pt}
\subsection{Laplacian system}\label{lapl}
We consider the system \eqref{eq:lti} proposed in \cite[Section 6]{dean2017} with
\[
    A = \begin{bmatrix}
        1.01 & 0.01 & 0 \\ 0.01 & 1.01 & 0.01 \\ 0 & 0.01 & 1.01
    \end{bmatrix}, B = I_{3}, C = \begin{bmatrix}
        1 & 0 & 0
    \end{bmatrix}, D = 0_{3 \times 1},
\]
which corresponds to a discrete-time, marginally unstable Laplacian system. The reference to track is $w_{\textrm{ref}}(t)=\col(0,0)$, which is a regulation problem. Figure \ref{fig:cons} shows regulation with varying levels ($\sigma=0.1, 0.2$) of measurement noise, comparing our proposed method with DD--MPC and the nominal regulator \eqref{eq:ls-behavioral}. It is observed that for $\gamma \geq 1.2$, the regulation performance is consistent in both noise cases. For the $\sigma=0.1$ case, similar performance is observed for ball radius $\rho \in (0, \sin(5^\circ)]$, beyond which we see significant drifting around the reference, without the asymptotic behavior.  Increasing the noise level to $\sigma=0.2$, we note that the same trend holds for $\rho \in (0, \sin(6.5^\circ)]$.

The performance is similar for $\sigma=0.1$ between the proposed algorithm, nominal controller and DD--MPC for $L=35$, subspaces in $\Gr(108, 140)$, and other parameters (see Supplementary). For $\sigma=0.2$ however, we observe slightly better regulation for the proposed algorithm compared to DD--MPC. The ball radius was updated from $\rho = \sin(2^\circ)$ to $\sin(3.5^\circ)$ due to the increase in noise level, while the slack variable $\lambda_{\alpha} \bar{\epsilon}$ from \citep{berberich2021} had to be increased from $0.1$ to $5$ for minimal drifting, to observe the regulation performance shown in Fig.~\ref{fig:cons}. 

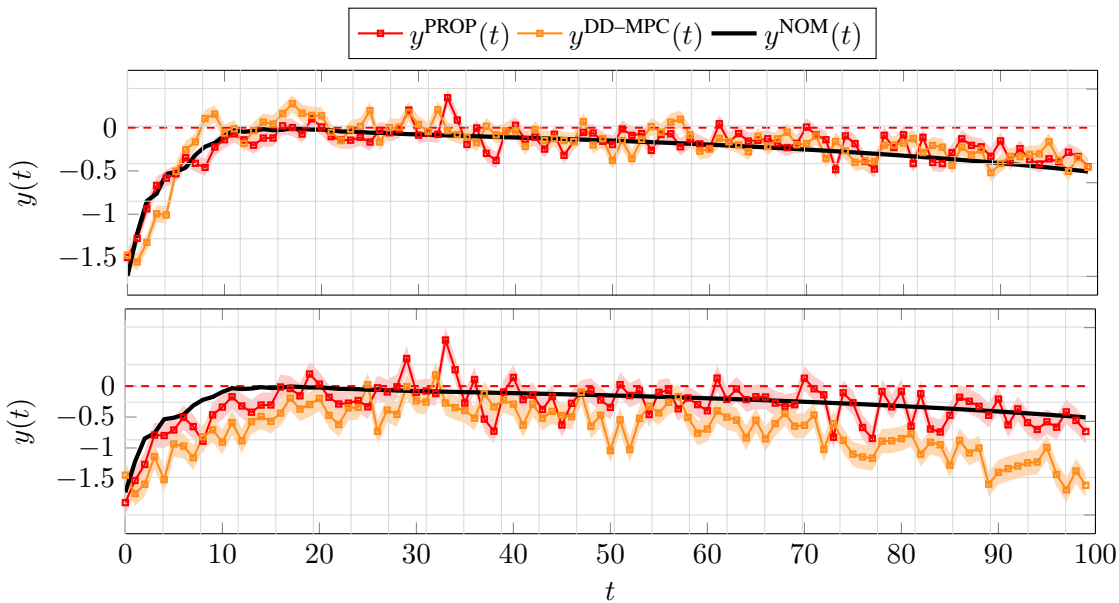
\begin{figure}[H]
  \begin{center}
    \begin{tikzpicture}
      \begin{axis}[
          width=0.95\linewidth,
          height=0.3\linewidth,
          xlabel={},
          xticklabels={},
          ylabel=$y(t)$ ,
          xmin = 0, xmax = 100,
          axis x line* = top,
          axis y line* = left,
          ytick={-1.5,-1,-0.5,0},
          yticklabels={$-1.5$,$-1$,$-0.5$,$0$},
          legend style={at={(0.5,1.05)}, anchor=south, legend columns=-1, fill=white, fill opacity=0.6, draw opacity=1,text opacity=1}
        ]
        
        \addplot[mark=square, mark size = 1, red, thick] table[x=t,y=ycl, col sep=comma]{data/case2_gamma4_rho2.csv};
        \addlegendentry{$y^{\textrm{PROP}}(t)$}
        \addplot[mark=square, mark size = 1, orange!90, thick] table[x=t,y=ycl, col sep=comma]{data/case2_mpc_sigma0p1.csv};
\addlegendentry{$y^{\textrm{DD--MPC}}(t)$}
         \addplot[mark=none, mark size=1, black, ultra thick] table[x=t,y=ycl_nl, col sep=comma]{data/case2_gamma4_rho2.csv};
    \addlegendentry{$y^{\textrm{NOM}}(t)$}
     \addplot [mark=none, red, dashed, thick]
        table[x=t,y expr=0,col sep=comma]{data/case2_gamma4_rho2.csv};
        \draw[step=0.5cm,gray!30,thin] (0,-3) grid (100,2);
        \addplot [name path=upper,draw=none] table[x=t,y expr=\thisrow{ycl}+0.1,col sep=comma]{data/case2_gamma4_rho2.csv};
        \addplot [name path=lower,draw=none] table[x=t,y expr=\thisrow{ycl}-0.1,col sep=comma]{data/case2_gamma4_rho2.csv};
        \addplot [fill=red!20] fill between[of=upper and lower];
        \addplot [name path=upper,draw=none] table[x=t,y expr=\thisrow{ycl}+0.1,col sep=comma]{data/case2_mpc_sigma0p1.csv};
        \addplot [name path=lower,draw=none] table[x=t,y expr=\thisrow{ycl}-0.1,col sep=comma]{data/case2_mpc_sigma0p1.csv};
        \addplot [fill=orange!30] fill between[of=upper and lower];
      \end{axis}
        \begin{axis}[
          width=0.95\linewidth,
          height=0.3\linewidth,
          xlabel={},
          xticklabels={},
          ylabel={},
          yticklabels={},
          xmin = 0, xmax = 100,
          axis x line* = bottom,
          axis y line* = right,
          legend style={at={(0.75,1.05)}, anchor=south, legend columns=-1, fill=white, fill opacity=0.6, draw opacity=1,text opacity=1}
        ]
      \end{axis}  
    \end{tikzpicture}
    \bigskip
    \vspace{-0.8cm}
    \begin{tikzpicture}
      \begin{axis}[
          width=0.95\linewidth,
          height=0.3\linewidth,
          xlabel=$t$,
          ylabel=$y(t)$ ,
          ytick={-1.5,-1,-0.5,0},
          yticklabels={$-1.5$,$-1$,$-0.5$,$0$},
          xtick={0,10,20,30,40,50,60,70,80,90,100},
          xticklabels={$0$,$10$,$20$,$30$,$40$,$50$,$60$,$70$,$80$,$90$,$100$},
          xmin = 0, xmax = 100,
          axis x line* = bottom,
          axis y line* = left,
          legend style={at={(0.85,0.75)}}
        ]
        \draw[step=0.5cm,gray!30,thin] (0,-3) grid (100,2);
        \addplot[mark=square, mark size = 1, red, thick] table[x=t,y=ycl, col sep=comma]{data/case2_gamma4_rho3p5.csv};
        \addplot [mark=none,red, dashed, thick]
        table[x=t,y expr=0,col sep=comma]{data/case2_gamma4_rho3p5.csv}; 
          \addplot[mark=square, mark size = 1, orange!90, thick] table[x=t,y=ycl, col sep=comma]{data/case2_mpc_sigma0p2.csv};
        \addplot [name path=upper,draw=none] table[x=t,y expr=\thisrow{ycl}+0.2,col sep=comma]{data/case2_gamma4_rho3p5.csv};
        \addplot [name path=lower,draw=none] table[x=t,y expr=\thisrow{ycl}-0.2,col sep=comma]{data/case2_gamma4_rho3p5.csv};
        \addplot [fill=red!20] fill between[of=upper and lower];
        \addplot [name path=upper,draw=none] table[x=t,y expr=\thisrow{ycl}+0.2,col sep=comma]{data/case2_mpc_sigma0p2.csv};
        \addplot [name path=lower,draw=none] table[x=t,y expr=\thisrow{ycl}-0.2,col sep=comma]{data/case2_mpc_sigma0p2.csv};
        \addplot [fill=orange!30] fill between[of=upper and lower];
        \addplot[mark=none, mark size=1, black, ultra thick] table[x=t,y=ycl_nl, col sep=comma]{data/case2_gamma4_rho2.csv};
      \end{axis}
     \begin{axis}[
          width=0.95\linewidth,
          height=0.3\linewidth,
          xlabel={},
          ylabel={},
          yticklabels={},
          xticklabels={},
          xmin = 0, xmax = 100,
          axis x line* = top,
          axis y line* = right,
          legend style={at={(0.98,0.75)}}
        ]
      \end{axis}  
    \end{tikzpicture}
  \end{center}
    \caption{Performance comparison of proposed method (PROP) with DD--MPC, for linear quadratic tracking of the Laplacian system in Section \ref{lapl} with reference $r=0$. Top: tracking for noise level $\sigma = 0.1$. Bottom: tracking for noise level $\sigma = 0.2$.}
  \label{fig:cons}
\end{figure}

\section{Conclusion}
\label{sec:conclusions} 
This work extended a geometrically grounded framework for least-squares to robust data-driven predictive control. By modeling behavioral subspace uncertainty as a metric ball on the Grassmannian manifold, we obtained a tractable min–max formulation. The core contribution is a closed-form solution to the inner maximization problem, leading to an efficient convex optimization algorithm. The resulting method enhances robustness in data-enabled predictive control, offering clear geometric interpretation and favorable computational scaling. The framework’s generality and efficiency make it promising for extensions to broader classes of data-driven estimation and control problems. An open question is whether the Grassmannian uncertainty ball admits a linear fractional representation in the sense of robust control theory, which could connect the behavioral uncertainty model to classical robustness analysis frameworks. Establishing formal closed-loop stability and robustness guarantees under bounded subspace uncertainty remains an important direction for future work and represents a natural extension of the geometric framework developed here.
\bibliography{ref}

\newpage

\appendix

\begin{center}
    \huge{Supplementary Sections}
    \[
    \
    \
    \
    \]
\end{center}

\section{Detailed Preliminaries}
\subsection*{Grassmann Manifold Geometry}

  The \textit{Grassmannian} $\Gr(k,n)$ is the Riemannian manifold of all $k$-dimensional linear subspaces of $\R^n$~\citep{boumal2023}.
A standard way to describe it is
through the \textit{Stiefel manifold}
\begin{equation} \label{eq:stiefel_manifold}
    \St(k,n)
    := \bigl\{\, Y \in \R^{n \times k} : Y^\top Y = I_k \,\bigr\},
\end{equation}
whose elements are matrices with orthonormal columns. The orthonormality
constraints $Y^\top Y = I_k$  define a smooth surface in $\R^{n \times k}$, so $\St(k,n)$ is a smooth
manifold of dimension $nk - \frac{k(k+1)}{2}$~\citep{boumal2023}. Different matrices 
in $\St(k,n)$ may span the same subspace, and identifying all orthonormal 
bases that generate the same subspace yields a smooth manifold structure.  
Equivalently, the Grassmannian can be viewed as the space of equivalence 
classes of $\St(k,n)$ under right multiplication by elements of $O(k)$, the group of all 
$k \times k$ orthogonal matrices. 

An alternative and particularly convenient representation uses orthogonal 
projection matrices. Every $k$-dimensional subspace $\mathcal{Y}$ of $\R^n$ admits a
unique orthogonal projector $P_{\mathcal{Y}}$, which makes this identification natural. 
In this representation, the Grassmannian can be written as
\[
    \Gr(k,n)
    =
    \bigl\{
        P \in \R^{n \times n}
        :\;
        P^\top = P,\;
        P^2 = P,\;
        \mathrm{rank}(P) = k
    \bigr\},
\]
so that each point corresponds to a unique symmetric, idempotent matrix of
rank~$k$, and  the associated $k$-dimensional subspace is given by the image of~$P$.

The link between these two representations is given by the canonical projection
\[
\pi:\St(k,n)\to\Gr(k,n),\qquad
\pi(Y)=YY^\top,
\]
which maps an orthonormal basis $Y$ to the orthogonal projector onto its span.
The map $\pi$ identifies all bases spanning the same subspace, since
$YQ$ and $Y$ yield the same projector for any $Q\in O(k)$. Equivalently, $\pi$ is constant along right–$O(k)$ orbits, that is, sets of the form $\{YQ:Q\in O(k)\}$, and thus becomes
injective on the quotient $\St(k,n)/O(k)$.

The Grassmannian admits a natural smooth Riemannian structure, inherited from
the canonical quotient metric on $\St(k,n)/O(k)$.  For any
$\mathcal{Y} \in \Gr(k,n)$ represented by $Y \in \St(k,n)$, the tangent space is
\[
T_{\mathcal{Y}}\Gr(k,n)
= \bigl\{\, V \in \R^{n\times k} \;\big|\; Y^\top V = 0 \,\bigr\}.
\]
where tangent directions are matrices orthogonal (column-wise) to $Y$,
reflecting that motion on the manifold preserves the orthonormality
constraint~\citep{boumal2023}.

Given a smooth function $f : \Gr(k,n) \to \R$, its Riemannian gradient at
$\mathcal{Y}$ is obtained by projecting the Euclidean gradient of any smooth
extension $\bar f : \R^{n\times k} \to \R$ onto the tangent space:
\[
\grad f \, (\mathcal{Y})
= P_{\mathcal{Y}}^\perp \,\nabla \bar f(Y),
\qquad
P_{\mathcal{Y}}^\perp := I_n - YY^\top,
\]
where $\bar f$ satisfies $\bar f(Y) = f(\mathcal{Y})$ for any orthonormal basis
$Y$ of $\mathcal{Y}$, and $\nabla \bar f(Y)$ denotes the standard Euclidean gradient of $\bar f$.

Several distances can be used to compare subspaces~\citep{boumal2023}, such as geodesic, projection, and Procrustes distances. The \emph{chordal distance} is particularly convenient because it admits a simple expression in terms of projection matrices and is computationally efficient. It is defined as
\[
d_2(\mathcal{Y}_1,\mathcal{Y}_2) := \sqrt{\mathrm{Tr}\!\left(P_{\mathcal{Y}_1}^\perp P_{\mathcal{Y}_2}\right)},
\]
and measures the discrepancy between the projectors. The chordal distance is invariant under changes of orthonormal bases and depends only on the principal angles between the subspaces, capturing how much one subspace deviates from the other.

The chordal metric is closely related to the \emph{gap} metric~\citep{boumal2023}, defined as
\[
d_\infty(\Y_1,\Y_2)
:=  \|P_{\Y_1} - P_{\Y_2}\|_2,
\]
which measures the largest principal angle between the subspaces. The two
metrics are equivalent (see \cite{padoan2022} and references therein) on $\Gr(k,n)$, and satisfy
\[
d_\infty(\Y_1,\Y_2)
\;\le\;
d_2(\Y_1,\Y_2)
\;\le\;
\sqrt{k}\, d_\infty(\Y_1,\Y_2),
\]
so the gap distance controls the worst–case angular deviation, whereas the
chordal distance aggregates the contribution of all principal angles. When average or aggregate separation is more informative than the largest principal
angle, the chordal metric provides a natural and tractable choice.

\section{Proof of Lemma \ref{lemma:1}}
\begin{proof} 
Fix $\Sy \in \mathbb{B}^d_{\rho}(\hat{\Sy})$. Then $P_{\Sy}$ is a constant symmetric idempotent matrix ($P_{\Sy}=P_{\Sy}^\top=P_{\Sy}^2$). Defining
$r(x):=P_{\Sy}x - b$, we have
\begin{equation*}
\begin{split}
f(x,\Sy)&=\|r(x)\|_2^2+\gamma\|M r(x)\|_2^2 \\
&= x^\top P_{\Sy}x - 2 b^\top P_{\Sy}x + \|b\|_2^2
  + \gamma\, x^\top P_{\Sy}M^\top M P_{\Sy}x - 2\gamma\, b^\top M^\top M P_{\Sy}x + \gamma\|Mb\|_2^2.
\end{split}
\end{equation*}
Since $P_{\Sy}\succeq 0$ and $P_{\Sy}M^\top M P_{\Sy}\succeq 0$, $f(\cdot,\Sy)$ is a convex quadratic in $x$.
Therefore $g(x)=\sup_{\Sy\in \mathbb{B}^d_{\rho}(\hat{\Sy})} f(x,\Sy)$ is the pointwise supremum of convex functions and is convex \cite[Sec.~3.2.3]{boyd2004}.

For the gradient, $f$ is continuously differentiable in $x$ for each $\Sy$, and 
$\mathbb{B}^d_{\rho}(\hat{\Sy})$ is compact in $\Gr(k,n)$; therefore, Danskin’s theorem applies
(e.g. \cite[Prop.~B.22]{bertsekas_nlp}). If the maximizer is unique,
\[
\nabla g(x) \;=\; \nabla_x f\bigl(x,\Sy^*(x)\bigr).
\]
Differentiating \eqref{eq:reg-cost} in $x$ (with $P_{\Sy}$ treated as constant) and using
$P_{\Sy}=P_{\Sy}^\top$ yields
\[
\nabla_x f(x,\Sy) \;=\; 2P_{\Sy}(P_{\Sy}x - b)
\;+\; 2\gamma\, P_{\Sy}M^\top M (P_{\Sy}x - b),
\]
which gives \eqref{eq:gradfx} after substituting $\Sy=\Sy^*(x)$. If the maximizer is not unique, Danskin's theorem further yields the stated subgradient characterization.
\end{proof}

\section{Proof of Theorem \ref{theorem:1} }
\begin{proof}
    The proof proceeds by reformulating the inner maximization into a sequence of equivalent problems, the last of which admits an explicit analytic solution. First, the inner problem in~\eqref{eq:reg-robust-lsq} can be simplified using the identity $\langle v_1, v_2 \rangle = \Tr(v_1 v_2^\top)$ for all $ v_1,v_2 \in \R^n$, where $\langle \cdot, \cdot \rangle$ is the standard inner product in Euclidean space. Using $P = P^2, MM^\top = I_l$ and ignoring constants which do not affect the optimizer, the inner problem in \eqref{eq:reg-robust-lsq} can be equivalently expressed as 
    \[
        \max_{\Sy \in \B_{\rho}^{d}(\hat{\Sy})} \langle P_{\Sy} x, (x-b) \rangle - \langle P_{\Sy} b, x \rangle + \gamma \langle P_{\Sy} M^\top MP_{\Sy} x, (x-b) \rangle - \gamma \langle P_{\Sy} M^\top Mb, x \rangle, \quad \forall x \in \R^n.
    \]
    Defining $A(x;\gamma) := xx^\top - xb^\top - bx^\top + \gamma(xx^\top - M^\top M bx^\top - xb^\top M^\top M)$, we obtain
   
    \begin{equation}\label{eq:const-robust-lsq1}
        \max_{\Sy \in \B^{d}_{\rho}(\hat{\Sy})} \Tr\left( P_{\Sy} A(x;\gamma) \right), \quad \forall x \in \R^n.
    \end{equation}
    We define the inequality constraint due to the ball $g(\Sy) := d(\Sy, {\hat{\Sy}})^2 - \rho^2 \leq 0.$
    Using the definition of $d$ in \eqref{eq:chordal}, we have $g(\Sy) = k - \Tr\left(P_{\Sy}P_{\hat{\Sy}} \right) - \rho^2$. The constrained problem \eqref{eq:const-robust-lsq1} can be solved by using an exact penalty function \citep{liu2019} :
    %
%
    \[  
        \max_{\Sy \in \Gr(k,n)} \Tr\left(P_{\Sy} A(x;\gamma)  \right) - \lambda g(\Sy) \quad \forall x \in \R^n
    \]
    for a sufficiently large constraint multiplier $\lambda \in \R_{\geq 0}$, such that the local maximizer $\Sy^*$ to the above problem is also the local maximizer of \eqref{eq:const-robust-lsq1}. Replacing $P_{\Sy}=YY^\top, P_{\hat{\Sy}} = \hat{Y}\hat{Y}^\top$ produces
    \[
    \max_{Y \in \St(k,n)} \Tr\left(Y^\top \left[A(x;\gamma) + \lambda\hat{Y} \hat{Y}^\top \right] Y \right).
    \]
    Defining $B(x,\lambda; \gamma) := A(x;\gamma) + \lambda \hat{Y} \hat{Y}^\top$, we obtain 
    \begin{equation}\label{eq:equiv-reg-rob1}
        \max_{Y \in \St(k,n)} \Tr\left(Y^\top B(x,\lambda;\gamma)Y \right).
    \end{equation}
    Alternatively, considering the inner problem in \eqref{eq:reg-robust-lsq}, we define the Lagrangian function $\mathcal{L}(x,\Sy, \lambda) := f(x,\Sy) - \lambda g(\Sy)$, for $\lambda \in \R_{\geq 0}$, such that we equivalently have: 
    \[
        \max_{\Sy \in \Gr(k,n)} \min_{\lambda} \mathcal{L}(x, \Sy, \lambda) = \max_{\Sy \in \Gr(k,n)} \min_{\lambda}  f(x,\Sy) - \lambda g(\Sy), \quad \text{for some } x \in \R^n.
    \]
    From first-order optimality conditions \citep[Def.~2.3]{liu2019}, we have
    \[
        \grad_{\Sy} \mathcal{L}(x, \Sy^*, \lambda^*) = 0  \implies \grad_{\Sy} f(x, \Sy^*) = \lambda^* \grad_{\Sy} g(\Sy^*),
    \]
    where $\Sy^*$ and $ \lambda^*$ are optimal. Taking the Riemannian gradient (see Section~\ref{sec:prelims}), we obtain
    \begin{equation}\label{eq:bmatrix1}
    \begin{array}{lll}
        \left(I-Y^*{Y^*}^\top \right)A(x;\gamma)Y^* = \lambda^*\left(I-Y^*{Y^*}^\top\right)\hat{Y}\hat{Y}^\top Y^* \\
        \implies B(x, \lambda^*;\gamma) Y^* = Y^*\left[{Y^*}^\top B(x, \lambda^*;\gamma)Y^*\right],
    \end{array}
    \end{equation}
     where $B(x, \lambda^*;\gamma) = A(x;\gamma) + \lambda^* \hat{Y}\hat{Y}^\top$. The stationary condition in \eqref{eq:bmatrix1} is the well-known eigenvalue problem \cite[Section 2.1]{absil} with the explicit solution
    \begin{equation*}
        Y^*(x) = \{\textrm{top--$k$ eigenvectors of }B(x, \lambda^*; \gamma) \}. 
    \end{equation*}
    This claim is also supported by the problem structure observed in~\eqref{eq:equiv-reg-rob1}. The corresponding optimal subspace $\Sy^*(x)$ is such that $P_{\Sy^*(x)} = \pi(Y^*(x))$.
\end{proof}

\fin{
\begin{remark}[Effective computation of $Y^*(x)$]
    The top-$k$ eigenspace of the matrix $B(x,\lambda;\gamma) = A(x;\gamma) + \lambda P_{\Syhat}$ can be computed effectively using a low-rank eigendecomposition technique since $\rank(B) \leq k+2$. This can bring down the computational cost from $\mathcal{O}(n^3)$ to $\mathcal{O}(n(k+2)^2)$.
\end{remark}

\begin{remark}[Uniqueness of $\Sy^*(x)$]
    There exists a threshold $\lambda^{**}$ such that for all $\lambda^* \geq \lambda^{**}$, the $k$-th and $(k+1)$-th eigenvalues of $B(x,\lambda^*;\gamma)$ are separated, ensuring that the top-$k$ eigenspace is unique. This follows from matrix perturbation theory \citep[Chapter VII]{bhatia1997matrix}: viewing $\lambda P_{\Syhat}$ as a nominal matrix perturbed by $A(x;\gamma)$, sufficiently large $\lambda$ dominates the perturbation and the top-$k$ eigenspace is a small perturbation of $\Syhat$, which is $k$-dimensional with a spectral gap. Hence, for the $\lambda^*$ found by Algorithm \ref{alg:cons-robust-ls} (via bisection), the solution $\Sy^*(x)$ is unique.
\end{remark}
}

\section{Additional Simulations Details}
The simulations are performed to show similar performance of our min-max, robust, direct data-driven controller with robust MPC (indirect) data-driven control~\citep{berberich2021}. Both algorithms are implemented in \texttt{MATLAB}, where PROP uses custom code for gradient descent, and DD--MPC uses {\texttt{YALMIP}} toolbox \fin{(\texttt{quadprog} solver)}, available at \url{https://yalmip.github.io}, to solve the problem as a quadratic programming problem. 

\paragraph{DD--MPC.} The DD--MPC algorithm proposes to solve the problem (subject to additional constraints):
\begin{equation*}
\begin{split}
    \min_{\alpha(t), \sigma(t), \bar{u}(t), \bar{y}(t)} & \quad \ell(\bar{u}_k(t), \bar{y}_k(t)) + \lambda_{\alpha} \bar{\epsilon} \|\alpha(t)\|_2^2 + \lambda_{\sigma} \| \sigma(t) \|_2^2  \\
    \textrm{subject to} & \quad \| \sigma_k(t) \|_{\infty} \leq \bar{\epsilon}(1+\|\alpha(t) \|_1)
\end{split}
\end{equation*}
where 
\[
    \ell(\bar{u}, \bar{u}) = \|\bar{u} - u^s \|_Q^2 + \| \bar{y} - y^s \|_R^2,
\]
for $Q,R \succ 0$, $(u^s, y^s)$ is the desired equilibrium, $\sigma$ is a slack variable to account for noisy measurements (robustness) and $\lambda_{\alpha}\bar{\epsilon}, \lambda_{\sigma} > 0$ are regularization weights which depend on the noise level.

\paragraph{Parameters.} The parameters for our simulations (PROP) are tabulated in Tables \ref{tab:prop_doub} and \ref{tab:prop_lap}.
\begin{table}[h]
    \centering
    \begin{minipage}{0.48\linewidth}
    \centering
    \begin{tabular}{c|c}
     \textbf{Parameter} & \textbf{Value} \\
     \hline 
    $T_{\textrm{ini}}$ & $10$ \\
    $T_{\textrm{f}}$ & $25$ \\
    $L$ & $35$ \\
    $n$ & $70$ \\
    $k$ & $37$ \\
    $\gamma$ & $4$ \\
    $\texttt{tolx}$ & $10^{-6}$ \\
    $\alpha$ & $0.1$ \\
    $T$ & $115$ \\
    \hline
    \end{tabular}
    \caption{Parameters for double integrator}
    \label{tab:prop_doub}
    \end{minipage}
    \begin{minipage}{0.48\linewidth}
    \centering
    \begin{tabular}{c|c}
     \textbf{Parameter} & \textbf{Value} \\
     \hline 
    $T_{\textrm{ini}}$ & $10$ \\
    $T_{\textrm{f}}$ & $25$ \\
    $L$ & $35$ \\
    $n$ & $140$ \\
    $k$ & $108$ \\
    $\gamma$ & $4$ \\
    $\texttt{tolx}$ & $10^{-6}$ \\
    $\alpha$ & $0.1$ \\
    $T$ & $150$ \\
    \hline
    \end{tabular}
    \caption{Parameters for Laplacian system.}
    \label{tab:prop_lap}
    \end{minipage}
\end{table}

Parameters for simulations of DD--MPC~\citep{berberich2021} are given in Tables \ref{tab:doub} and \ref{tab:lap}.

\begin{table}[h]
    \centering
    \begin{minipage}{0.48\linewidth}
    \centering
    \begin{tabular}{c|c}
     \textbf{Parameter} & \textbf{Value} \\
     \hline 
    $T_{\textrm{ini}}$ & $10$ \\
    $T_{\textrm{f}}$ & $25$ \\
    $L$ & $35$ \\
    $Q$ & $I_2$ \\
    $R$ & $0.5$ \\
    $\lambda_{\sigma}$ & $10^3$\\
    $T$ & $115$ \\
    \hline
    \end{tabular}
    \caption{Parameters for double integrator.}
    \label{tab:doub}
    \end{minipage}
    \begin{minipage}{0.48\linewidth}
    \centering
    \begin{tabular}{c|c}
     \textbf{Parameter} & \textbf{Value} \\
     \hline 
    $T_{\textrm{ini}}$ & $10$ \\
    $T_{\textrm{f}}$ & $25$ \\
    $L$ & $35$ \\
    $Q$ & $0.1 I_3$ \\
    $R$ & $5$ \\
    $\lambda_{\sigma}$ & $10^3$\\
    $T$ & $150$ \\
    \hline
    \end{tabular}
    \caption{Parameters for Laplacian system.}
    \label{tab:lap}
    \end{minipage}
\end{table}

\paragraph{Results.} The control inputs (PROP) for the Laplacian system corresponding to the outputs shown in Figure~\ref{fig:cons} are shown below in Figure~\ref{fig:optimal_inputs_lqr} for completeness. 
For both case studies, we present the optimization convergence plots for double integrator in Figure~\ref{fig:dint_conv} and for the Laplacian system in Figure~\ref{fig:lapl_conv}, by examining the performance of Algorithm~\ref{alg:cons-robust-ls} at some time-step during the simulation. 

\begin{figure}[h]
  \begin{center}
    \begin{tikzpicture}
      \begin{axis}[
          width=0.95\linewidth,
          height=0.25\linewidth,
          xlabel={},
          xticklabels={},
          ylabel=$u(t)$ ,
          xmin = 0, xmax = 100,
          axis x line* = top,
          axis y line* = left,
          legend style={at={(0.78,0.65)}, anchor=south, legend columns=-1, fill=white, fill opacity=0.6, draw opacity=1,text opacity=1}
        ]
        \draw[step=0.5cm,gray!30,thin] (0,-3) grid (100,2);
        \addplot[mark=*, mark size = 0.5, blue, thick] table[x=t,y=ucl1, col sep=comma]{data/case2_gamma4_rho2.csv};
        \addlegendentry{$u_1(t)$}
        \addplot[mark=*, mark size = 0.5, green, thick] table[x=t,y=ucl2, col sep=comma]{data/case2_gamma4_rho2.csv};
        \addlegendentry{$u_2(t)$}
        \addplot[mark=*, mark size = 0.5, orange, thick] table[x=t,y=ucl3, col sep=comma]{data/case2_gamma4_rho2.csv};
        \addlegendentry{$u_3(t)$}
      \end{axis}
        \begin{axis}[
          width=0.95\linewidth,
          height=0.25\linewidth,
          xlabel={},
          xticklabels={},
          ylabel={},
          yticklabels={},
          xmin = 0, xmax = 100,
          axis x line* = bottom,
          axis y line* = right,
          legend style={at={(0.75,1.05)}, anchor=south, legend columns=-1, fill=white, fill opacity=0.6, draw opacity=1,text opacity=1}
        ]
      \end{axis}  
    \end{tikzpicture}
    \bigskip
    \vspace{-0.8cm}
    \begin{tikzpicture}
      \begin{axis}[
          width=0.95\linewidth,
          height=0.25\linewidth,
          xlabel=$t$,
          ylabel=$u(t)$ ,
          ytick={0,0.2},
          yticklabels={$0$,$0.2$},
          xmin = 0, xmax = 100,
          axis x line* = bottom,
          axis y line* = left,
          legend style={at={(0.85,0.75)}}
        ]
 \addplot[mark=*, mark size = 0.5, blue, thick] table[x=t,y=ucl1, col sep=comma]{data/case2_gamma4_rho3p5.csv};
        \draw[step=0.5cm,gray!30,thin] (0,-3) grid (100,2);
        \addplot[mark=*, mark size = 0.5, green, thick] table[x=t,y=ucl2, col sep=comma]{data/case2_gamma4_rho3p5.csv};
        \addplot[mark=*, mark size = 0.5, orange, thick] table[x=t,y=ucl3, col sep=comma]{data/case2_gamma4_rho3p5.csv};
      \end{axis}
     \begin{axis}[
          width=0.95\linewidth,
          height=0.25\linewidth,
          xlabel={},
          ylabel={},
          yticklabels={},
          xticklabels={},
          xmin = 0, xmax = 100,
          axis x line* = top,
          axis y line* = right,
          legend style={at={(0.98,0.75)}}
        ]
      \end{axis}  
    \end{tikzpicture}
  \caption{Optimal control inputs for LQR of Laplacian system. Top: tracking for measurement noise level $\sigma = 0.1$. Bottom: tracking for measurement noise level $\sigma = 0.2$.}
  \label{fig:optimal_inputs_lqr}
  \end{center}
\end{figure}
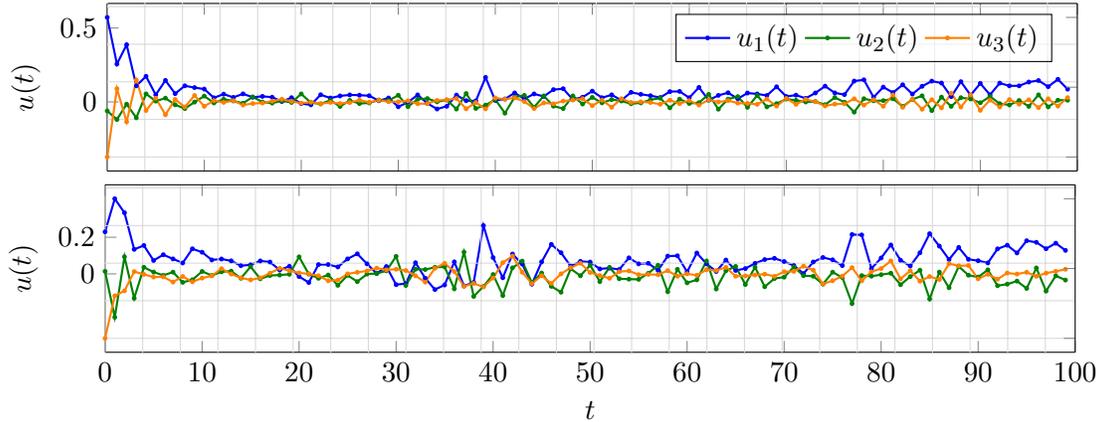

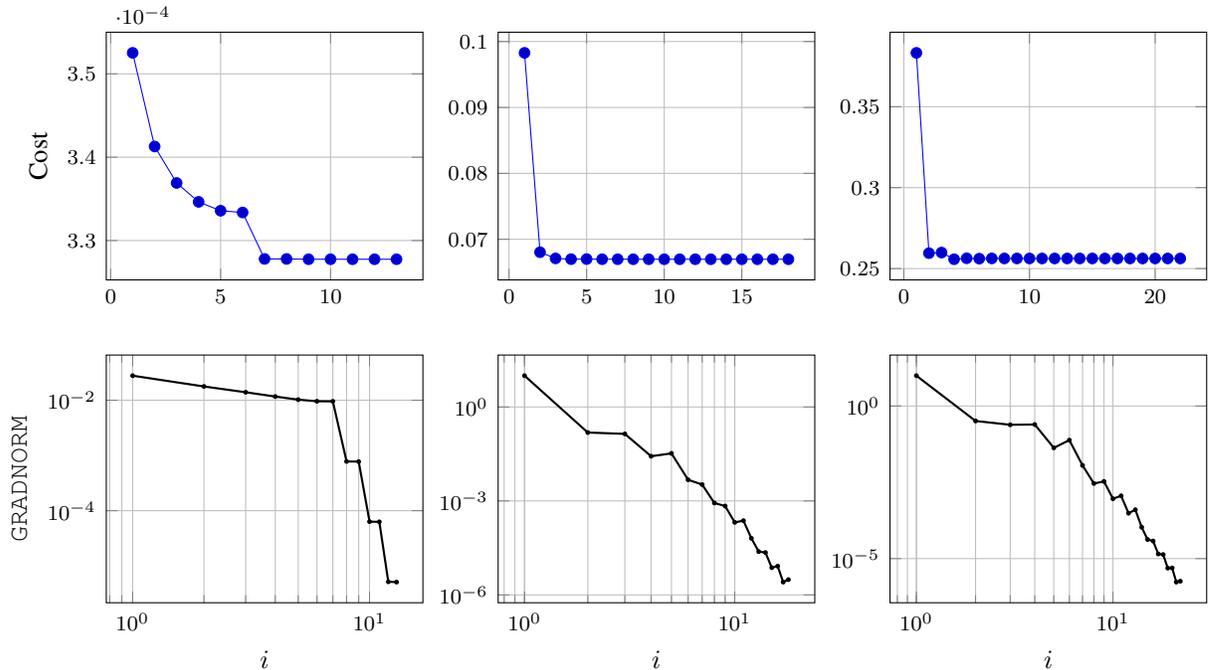
\begin{figure}[h]
\centering
\begin{tikzpicture}

\begin{groupplot}[
    group style={
        group size=3 by 2,
        horizontal sep=1cm,
        vertical sep=1cm
    },
    width=0.38\textwidth,
    height=0.32\textwidth,
    xlabel={},
    ylabel style={font=\small},
    xlabel style={font=\small},
    ticklabel style={font=\scriptsize},
    grid=both,
]

\nextgroupplot[ylabel={Cost}]
\addplot table[x=k0,y=cost0,col sep=comma] {data/case1_cost_gradnorm.csv};

\nextgroupplot[ylabel={}, ytick={0.07,0.08,0.09,0.1}, yticklabels={$0.07$,$0.08$,$0.09$,$0.1$}]
\addplot table[x=k1,y=cost1,col sep=comma] {data/case1_cost_gradnorm.csv};

\nextgroupplot[xlabel={},ylabel={}]
\addplot table[x=k2,y=cost2,col sep=comma] {data/case1_cost_gradnorm.csv};

\nextgroupplot[ylabel={\texttt{GRADNORM}},xlabel={$i$}, xmode=log, ymode=log]
\addplot[mark=*,mark size=0.5, black, thick] table[x=k0,y=grad0,col sep=comma] {data/case1_cost_gradnorm.csv};

\nextgroupplot[ylabel={},xlabel={$i$}, xmode=log, ymode=log]
\addplot[mark=*,mark size=0.5, black, thick] table[x=k1,y=grad1,col sep=comma] {data/case1_cost_gradnorm.csv};

\nextgroupplot[xlabel={$i$},ylabel={}, xmode=log, ymode=log]
\addplot[mark=*,mark size=0.5, black, thick] table[x=k2,y=grad2,col sep=comma] {data/case1_cost_gradnorm.csv};

\end{groupplot}
\end{tikzpicture}
\caption{Cost $f(x_i,\Sy^*(x_i))$ and gradnorm $\|\nabla_x f(x_i,\Sy^*(x_i)) \|$ evolution with iteration index $i$, at $t=50$. Left: (NOM), Center: (PROP) for $\sigma=0.1$, Right: (PROP) for $\sigma=0.2$.}
\label{fig:dint_conv}
\end{figure}

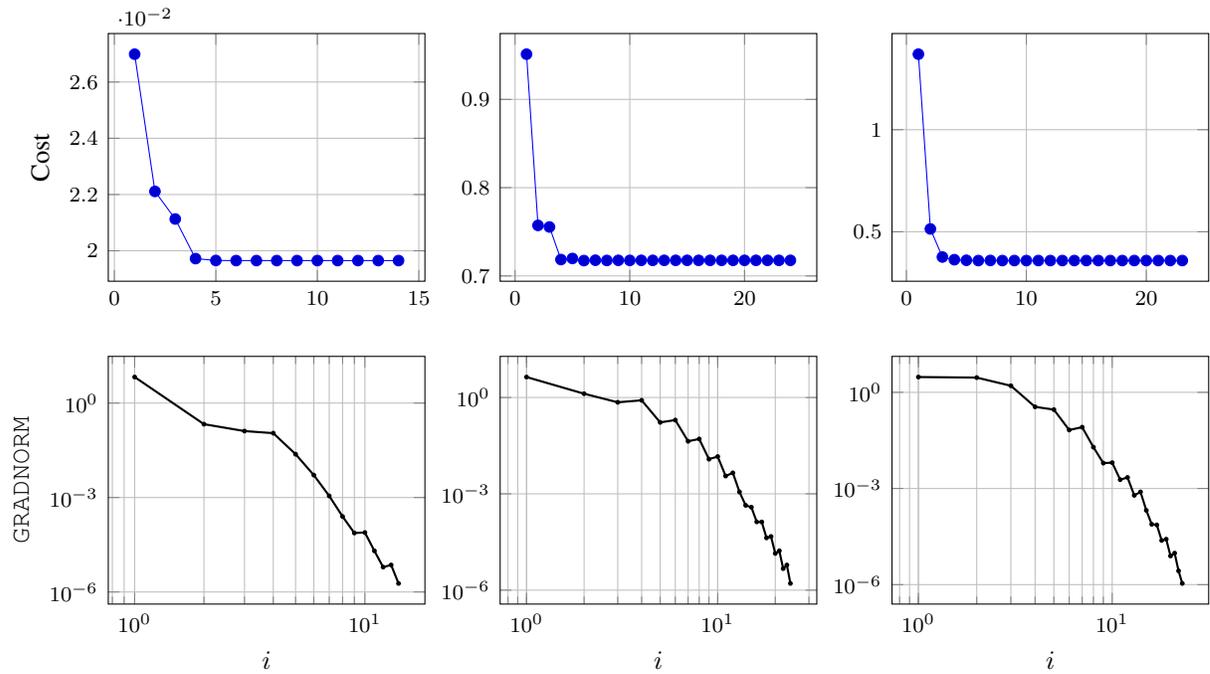
\begin{figure}[h]
\centering
\begin{tikzpicture}
\begin{groupplot}[
    group style={
        group size=3 by 2,
        horizontal sep=1cm,
        vertical sep=1cm
    },
    width=0.38\textwidth,
    height=0.32\textwidth,
    xlabel={},
    ylabel style={font=\small},
    xlabel style={font=\small},
    ticklabel style={font=\scriptsize},
    grid=both,
]

\nextgroupplot[ylabel={Cost}]
\addplot table[x=k0,y=cost0,col sep=comma] {data/case2_cost_gradnorm.csv};

\nextgroupplot[ylabel={}]
\addplot table[x=k1,y=cost1,col sep=comma] {data/case2_cost_gradnorm.csv};

\nextgroupplot[xlabel={},ylabel={}]
\addplot table[x=k2,y=cost2,col sep=comma] {data/case2_cost_gradnorm.csv};

\nextgroupplot[ylabel={\texttt{GRADNORM}},xlabel={$i$}, xmode=log, ymode=log]
\addplot[mark=*,mark size=0.5, black, thick] table[x=k0,y=grad0,col sep=comma] {data/case2_cost_gradnorm.csv};

\nextgroupplot[ylabel={},xlabel={$i$}, xmode=log, ymode=log]
\addplot[mark=*,mark size=0.5, black, thick] table[x=k1,y=grad1,col sep=comma] {data/case2_cost_gradnorm.csv};

\nextgroupplot[xlabel={$i$},ylabel={}, xmode=log, ymode=log]
\addplot[mark=*,mark size=0.5, black, thick] table[x=k2,y=grad2,col sep=comma] {data/case2_cost_gradnorm.csv};

\end{groupplot}
\end{tikzpicture}
\caption{Cost $f(x_i,\Sy^*(x_i))$ and gradnorm $\|\nabla_x f(x_i,\Sy^*(x_i)) \|$ evolution with iteration index $i$, at $t=10$. Left: (NOM), Center: (PROP) for $\sigma=0.1$, Right: (PROP) for $\sigma=0.2$.}
\label{fig:lapl_conv}
\end{figure}

\clearpage
Figure~\ref{fig:dint_constraints} shows the evolution of the constraint multiplier $\lambda^*$ for the double integrator. We observe that as the trajectories asymptotically track the reference, the penalty $\lambda^*$ decreases (roughly) asymptotically to zero. Similar trends are observed for the Laplacian system as seen in Figure~\ref{fig:lapl_constraints}.

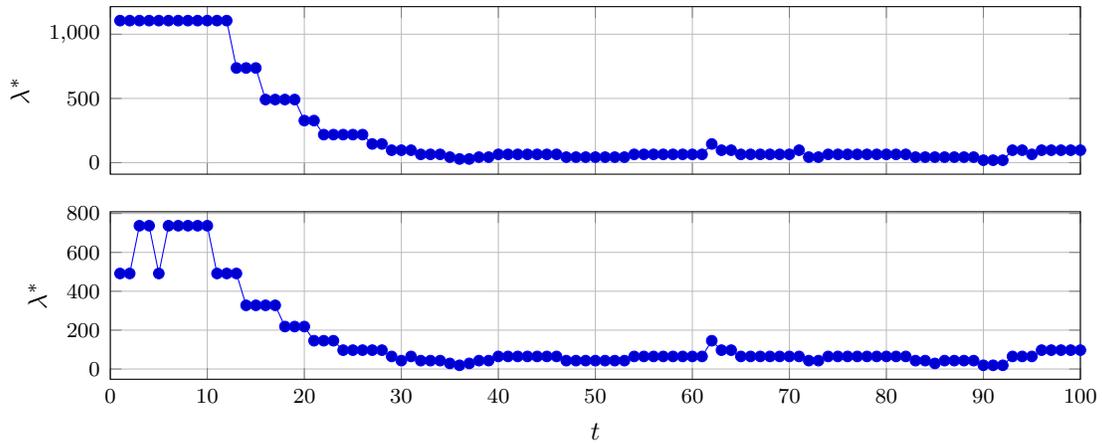
\begin{figure}[h]
\centering
\begin{tikzpicture}

\begin{groupplot}[
    group style={
        group size=1 by 2,
        horizontal sep=1cm,
        vertical sep=0.5cm
    },
    width=0.95\textwidth,
    height=0.25\textwidth,
    xlabel={},
    xmin=0, xmax=100,
    ylabel style={font=\small},
    xlabel style={font=\small},
    ticklabel style={font=\scriptsize},
    grid=both,
]

\nextgroupplot[ylabel={$\lambda^*$}, xticklabels={}]
\addplot table[x=k11,y=lambda11,col sep=comma] {data/case12_constraints.csv};

\nextgroupplot[ylabel={$\lambda^*$}, xlabel={$t$}]
\addplot table[x=k12,y=lambda12,col sep=comma] {data/case12_constraints.csv};

\end{groupplot}
\end{tikzpicture}
\caption{Constraint multiplier $\lambda^*$ at each time-step of the double-integrator reference tracking. Top and bottom figures corresponding to noise levels $\sigma=0.1$ and $\sigma=0.2$ respectively.}
\label{fig:dint_constraints}
\end{figure}

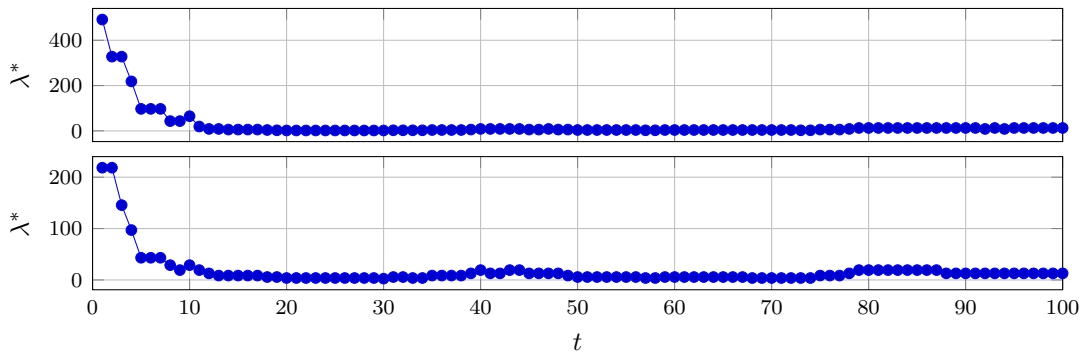
\begin{figure}[h]
\centering
\begin{tikzpicture}

\begin{groupplot}[
    group style={
        group size=1 by 2,
        horizontal sep=1cm,
        vertical sep=0.2cm
    },
    width=0.95\textwidth,
    height=0.22\textwidth,
    xlabel={},
    xmin=0, xmax=100,
    ylabel style={font=\small},
    xlabel style={font=\small},
    ticklabel style={font=\scriptsize},
    grid=both,
]

\nextgroupplot[ylabel={$\lambda^*$}, xticklabels={}]
\addplot table[x=k21,y=lambda21,col sep=comma] {data/case12_constraints.csv};

\nextgroupplot[ylabel={$\lambda^*$}, xlabel={$t$}]
\addplot table[x=k22,y=lambda22,col sep=comma] {data/case12_constraints.csv};

\end{groupplot}
\end{tikzpicture}
\caption{Constraint multiplier $\lambda^*$ at each time-step of the Laplacian system regulation. Top and bottom figures corresponding to noise levels $\sigma=0.1$ and $\sigma=0.2$ respectively.}
\label{fig:lapl_constraints}
\end{figure}

\subsection*{Comparison with TSRGDA \citep{coulson2025}}
We compare the convergence performance of our proposed algorithm with \citep{coulson2025} which uses timescale-separated Riemannian gradient descent-ascent (TSRGDA) algorithm to solve the robust least-squares problem. Here we only examine the robust least-squares problem \eqref{eq:least_squares_robust_geometric}. Consider subspaces in $\Gr(37,70)$,   $\rho = \sin(\pi/8)$ and
    $\hat{\Sy} = \Image \begin{bmatrix}
        I_{37} \\ 0_{33 \times 37}
    \end{bmatrix}$. The parameters used for TSRGDA are shown in Table \ref{tab:tsrgda}.

\begin{table}[h]
    \centering
    \begin{tabular}{c|c}
     \textbf{Parameter} & \textbf{Value} \\
     \hline 
    $\lambda$ & $1000$ \\
    $u$ & $0.1$ \\
    $\eta_x$ & $10^{-4}$ \\
    $\eta_y$ & $10^{-3}$ \\
    \hline
    \end{tabular}
    \caption{Parameters for TSRGDA.}
    \label{tab:tsrgda}
\end{table}

From Figure~\ref{fig:convergence_comp}, we observe that for the TSRGDA algorithm, it takes $\sim 10^5$ iterations for the gradnorm to reach within a tolerance of $10^{-5}$ whereas the proposed algorithm reaches it within $\sim 10^3$ iterations. It is also crucial to note that for the proposed algorithm, due to the explicit solution of $\Sy^*(x_i)$, we observe that $\|\grad_{\Sy} f(x_i,\Sy^*(x_i)) \| \sim 10^{-14}$ for all iterates $i=0,1,\ldots$. This is the crucial novelty of the proposed algorithm, wherein the min-max setup is elegantly converted to a convex optimization framework. The run time for the proposed method for the gradnorm (w.r.t $x$) to reach $\texttt{tolx} = 10^{-4}$ is $3.2$ seconds, whereas the TSRGDA method takes $25$ seconds. Increasing the step-sizes $\eta_x,\eta_y$ leads to negative effect on convergence. Thus, along with tractability and efficiency, the proposed method also guarantees convergence for $\alpha \in (0, 1]$.  

\begin{figure}[h]
\centering
\begin{tikzpicture}

\begin{groupplot}[
    group style={
        group size=1 by 1,
        horizontal sep=1cm,
        vertical sep=1cm
    },
    width=0.95\textwidth,
    height=0.42\textwidth,
    xlabel={},
    xmin=1, xmax=100000,
    ylabel style={font=\small},
    xlabel style={font=\small},
    ticklabel style={font=\scriptsize},
    grid=both,
    legend style={at={(0.48,1.05)}, anchor=south, legend columns=-1, fill=white, fill opacity=0.6, draw opacity=1,text opacity=1},
]

\nextgroupplot[xmode=log, ymode=log, ylabel={\texttt{GRADNORM}}]
\addplot[mark=square, red, mark size=2] table[x=k1,y=gradx1,col sep=comma] {data/comp_tsrgda_prop.csv};
\addlegendentry{$\|\nabla_x f(x,\Sy)\|^{\texttt{TSRGDA}}$}
\addplot[mark=square, blue, mark size=2] table[x=k1,y=grady1,col sep=comma] {data/comp_tsrgda_prop.csv};
\addlegendentry{$\|\grad_\Sy f(x,\Sy) \|^{\texttt{TSRGDA}}$}

\addplot[mark=diamond, red, mark size=3] table[x=k2,y=gradx2,col sep=comma] {data/comp_tsrgda_prop.csv};
\addlegendentry{$\|\nabla_x f(x,\Sy)\|^{\texttt{PROP}}$}
\addplot[mark=diamond, blue, mark size=3] table[x=k2,y=grady2,col sep=comma] {data/comp_tsrgda_prop.csv};
\addlegendentry{$\|\grad_\Sy f(x,\Sy) \|^{\texttt{PROP}}$}

\end{groupplot}
\end{tikzpicture}
\caption{Gradnorm comparison between TSRGDA and proposed method (PROP).}
\label{fig:convergence_comp}
\end{figure}
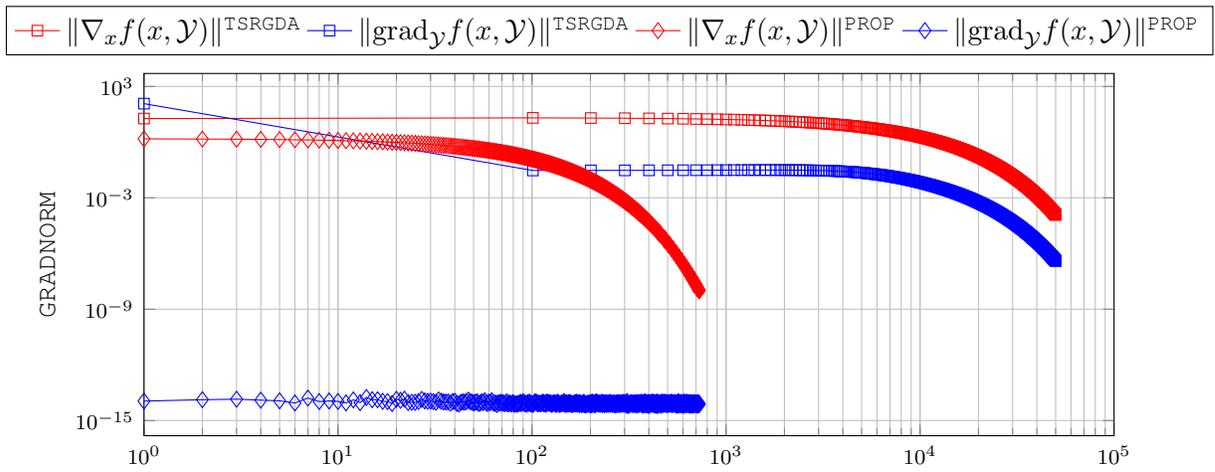 
\end{document}